\documentclass[12pt,a4paper,amsfonts,amscd,amssymb]{amsart}
\usepackage[all, cmtip]{xy}
\usepackage{amsmath}
\usepackage[utf8]{inputenc}
\usepackage[T1]{fontenc}

\usepackage{wasysym}

\theoremstyle{plain} 

\begingroup

\newtheorem{theorem}{Theorem}[section]

\newtheorem{cor}[theorem]{Corollary}

\newtheorem{lemma}[theorem]{Lemma}

\newtheorem{prop}[theorem]{Proposition}

\endgroup

\theoremstyle{definition}

\newtheorem{definition}[theorem]{Definition}

\theoremstyle{remark}

\begin{document}

\title[On the $n$-transitivity of the group of equivariant diffeomorphisms]
{On the $n$-transitivity of the group of equivariant diffeomorphisms}

\author{Marja Kankaanrinta}

\address{Tampere University}
\email{marja.kankaanrinta@tuni.fi}

\date{\today}


\keywords{}

\begin{abstract}  
Let $G$ be a Lie group and let $M$ be a proper smooth $G$-manifold. 
If $M$ is connected and $\dim(M)\geq 2$,  the
group of diffeomorphisms of $M$, that are isotopic to the identity through
a compactly supported isotopy,
acts $n$-transitively on $M$, for any $n$.  In this paper, we prove a version of the
$n$-transitivity result for the group of equivariant diffeomorphisms of $M$. As a corollary we obtain a
result concerning  diffeomorphisms of the orbit space $M/G$.  A special case of the result for orbit spaces
gives an $n$-transitivity result for orbifold diffeomorphisms that was earlier proved by F. Pasquotto and T. O. Rot.
\end{abstract}

\maketitle

\section{Introduction}
\label{intro}

\noindent Let $M$ be a connected smooth manifold of dimension at least two.  A subgroup $H$ of the diffeomorphism
group ${\rm Diff}(M)$ is said to act $n$-{\it transitively} on $M$, if for any two $n$-tuples of pairwise distinct points 
$(x_1,\ldots, x_n)$ and $(y_1,\ldots, y_n)$ in $M$ there is an element $f\in H$ such that
$f(x_i)=y_i$, for all $i\in \{1,\ldots, n\}$.
In \cite{MV}, P. Michor and C. Vizman
studied the $n$-transitivity of actions of various subgroups of  ${\rm Diff}(M)$ on $M$. In
particular, they proved that the group ${\rm Diff}_c(M)$ of  diffeomorphisms of $M$, 
that are isotopic to the identity through a compactly supported isotopy,
acts $n$-transitively on $M$, for all $n\in{\mathbb{N}}$.

Assume a Lie group $G$ acts smoothly on $M$. It is obvious that in general the group ${\rm Diff}^G(M)$ of
equivariant diffeomorphisms of $M$ cannot act transitively on $M$. For example, if $f\colon M\to M$ is a $G$-equivariant
diffeomorphism, then the isotropy groups at the points $x$ and $f(x)$ must be the same for all $x\in M$. Also, the value 
of an equivariant map at $gx$, for any $g\in G$, is already determined by the value of the map at $x$. Taking into account
these restrictions we obtain an equivariant version of the $n$-transitivity result:

\begin{theorem}
\label{result1}
Let $G$ be a Lie group and let $M$ be a proper smooth $G$-manifold  with orbit space $M/G$. 
Let
$n\in{\mathbb{N}}$, and let $(x_1,\ldots, x_n)$ and $(y_1,\ldots, y_n)$ be two $n$-tuples of points in
$M$. Assume the following conditions are satisfied:
\begin{enumerate}
\item $(Gx_1,\ldots, Gx_n)$ and $(Gy_1,\ldots, Gy_n)$ are two $n$-tuples of pairwise distinct points in the orbit space $M/G$.
\item For every $i$, $1\leq i\leq n$, $G_{x_i}=G_{y_i}$ and $x_i$ and $y_i$ lie in the same connected component of 
$M_{G_{x_i}}=\{ x\in M\mid G_x=G_{x_i}\}$.
\item If $H$ is a compact subgroup of $G$, then any connected component $Y$ of $M_H$
satisfying $\dim\bigl( N(H)/H\bigr)=\dim(Y)-1$
contains at most one of the points $x_i$.
\end{enumerate}
Then there exists a smooth $G$-equivariant diffeomorphism $f\colon M\to M$
isotopic to the identity through a $G$-compactly supported isotopy, such that $f(x_i)=y_i$ for all $1\leq i\leq n$.
\end{theorem}

In \cite{PR}, F. Pasquotto and T. O. Rot proved a version of the $n$-transitivity result for orbifold diffeomorphisms. 
As in the equivariant case, one has to take into account the local groups at points of the orbifold. Since every smooth
 reduced, i.e., effective, orbifold can be considered as 
 a quotient of  a smooth, effective, almost free  action of a compact Lie group on a smooth manifold, 
 it follows that the $n$-transitivity results in the equivariant case and in the orbifold case are closely related. 
 Indeed,
 by considering diffeomorphisms of  orbit spaces, we  obtain an $n$-transitivity result for diffeomorphisms of the
 orbit space $M/G$, see Corollary \ref{cor1}. As a special case, we then obtain the $n$-transitivity result of Pasquotto and Rot for
 orbifold diffeomorphisms:

\begin{theorem}
\label{result2}
Let $X$ be a reduced smooth orbifold. Let
$n\in{\mathbb{N}}$, and let $(x_1,\ldots, x_n)$ and $(y_1,\ldots, y_n)$ be two $n$-tuples of pairwise distinct points in
$X$. Assume that for every $i$, $1\leq i\leq n$, $x_i$ and $y_i$ lie in the same connected component of
the singular strata of $X$, and that any one-dimensional connected component of the strata of $X$
contains at most one of the points $x_i$. Then there exists a smooth orbifold diffeomorphism $f\colon X\to
X$ isotopic to the identity through a compactly supported isotopy, such that $f(x_i)=y_i$, for all
$1\leq i\leq n$.
\end{theorem}

\section{Group actions}
\label{group}

\noindent Let $G$ be a Lie group acting smoothly on a smooth (i.e., ${\rm C}^\infty$) manifold $M$.
The action is called {\it proper}, if the map
$$
\phi\colon G\times M\to M\times M,\,\,\, (g,x)\mapsto (gx,x),
$$
is proper, i.e., if the inverse images of compact sets are compact. In this case we call $M$ a {\it proper smooth $G$-manifold}. 
Notice that $\phi$ is a closed map, if it is proper (see \cite{Pa2}).
In particular, every action of a compact Lie group is proper.

Let $x\in M$. We denote the {\it isotropy subgroup} $\{ g\in G\mid gx=x\}$ of $x$ by $G_x$. If $G$ acts properly on $M$, then the isotropy
subgroup of every $x\in M$ is compact. The action is called {\it free}, if $G_x=\{e\}$, for every $x\in M$,
and {\it almost free}, if $G_x$ is finite for every $x\in M$. If the identity element $e$ of $G$ is the only element of $G$ that fixes every point
of $M$, we call the action {\it effective}.

For a compact subgroup $H$
of $G$ we define
$$
M_H=\{ x\in M\mid G_x=H\},\quad M_{(H)}=\{ x\in M\mid G_x\sim H\}
$$
and
$$
M^H=\{ x\in M\mid H\subset G_x\},
$$
where $G_x\sim H$ means that $G_x=gHg^{-1}$, for some $g\in G$. If $M$ is a proper smooth $G$-manifold, then the sets
$M_H$, $M_{(H)}$ and $M^H$ are $\Sigma$-submanifolds of $M$. In other words, each connected component of $M_H$,
$M_{(H)}$ and $M^H$ is a submanifold of $M$.
The dimension of the connected components may vary.
The sets satisfy the equality
$$
M_H= M^H\cap M_{(H)},
$$
see Corollary 4.2.8 in \cite{P}.  The set $M^H$, being a fixed point set, is closed in $M$.

The {\it orbit} of $x\in M$ is 
$
Gx=\{ gx\mid g\in G\},
$
and the {\it orbit space} is denoted by $G/M$. A subset $A\subset M$ is called $G$-{\it compact}, if  $\pi(A)$ is compact, where
$\pi\colon M\to M/G$, $x\mapsto Gx$, denotes the {\it natural projection}.

Let $H$ be a closed subgroup of a Lie group $G$, and assume $M$ is a proper smooth $H$-manifold.
We define the {\it twisted product} $G\times_HM$ to be the orbit space of the smooth $H$-manifold
$G\times M$, where $H$ acts on $G\times M$ by $h(g,x)=(gh^{-1}, hx)$. Let $p\colon G\times M\to
G\times_HM$ be the natural projection, and denote $p(g,x)=[g,x]$. Then $G$ acts on
$G\times_HM$ by $g'[g,x]=[g'g,x]$ and $G\times_HM$ is a proper smooth $G$-manifold. Let $N$ be a
smooth $G$-manifold. Then any smooth $H$-equivariant map $f\colon M\to N$ induces a smooth
$G$-equivariant map
$$
\tilde{f}\colon G\times_HM\to N,\,\,\, [g,x]\mapsto gf(x).
$$
For basic properties of twisted products, see \cite{I} or \cite{IK}.

We will be using
the  differentiable slice theorem (Proposition 2.2.2. in \cite{Pa}) for proper actions:  
Let $x\in M$, and let $Gx$ be the orbit of $x$ in $M$.  
Equip a $G$-invariant
neighborhood of $x$
with a smooth $G$-invariant Riemannian metric.
Then $M$ has a smooth
$G_x$-invariant submanifold ${\rm N}_x$ containing $x$
that is $G_x$-equivariantly diffeomorphic
to an open $G_x$-invariant neighborhood of the origin in the normal space
${\rm T}_xM/{\rm T}_xGx$ to $Gx$ at $x$. The manifold ${\rm N}_x$ is called a 
{\it linear slice} (or just a {\it slice}) at $x$.  The exponential map takes an open 
$G$-invariant neighborhood of the 
zero section of the normal bundle of $Gx$ diffeomorphically to the open neighborhood 
$G{\rm N}_x$ of $Gx$.
The map
$$
G\times _{G_x}{\rm N}_x\to G{\rm N}_x,\, \, \,
[g,y]\mapsto gy,
$$
is a smooth $G$-equivariant diffeomorphism.  
Therefore,
we may identify $G{\rm N}_x$ with the twisted product
$G\times _{G_x}{\rm N}_x$. The map
$$
G{\rm N}_x\to G/G_x,\,\,\, gy\mapsto gG_x,
$$
is smooth and the inverse image of $eG_x$ equals ${\rm N}_x$.
The map 
$$
G\times {\rm N}_x\to G{\rm N}_x,\,\,\, (g,y)\mapsto gy,
$$
is open.

\begin{lemma}
\label{samat}
Let $G$ be a Lie group and let $H$ be a compact subgroup of $G$. Let $M$ be a proper smooth $G$-manifold. Assume
$x\in M$ and $G_x=H$. Let ${\rm N}$ be a slice at $x$. Then ${\rm N}_H={\rm N}_{(H)}={\rm N}^H$ and $(G{\rm N})_H=(G{\rm N})^H$.
\end{lemma}

\begin{proof} By Lemma 4.2.9 in \cite{P},  the only closed subgroup  of $H$ that is conjugate to $H$ is $H$ itself.
The claims follow, since for every $y\in {\rm N}$, $G_y\subset G_x=H$ and since $G_{gy}=gG_yg^{-1}$, for all 
$y\in {\rm N}$ and for all $g\in G$.
\end{proof}

\begin{lemma}
\label{G-paksu}
Let $G$ be a Lie group and let $H$ be a compact subgroup of $G$. Let $M$ be a proper smooth
$G$-manifold. Assume $G_x=H$ for all $x\in M$. Let $\phi\colon G\times M\to M\times M$, 
$(g,x)\mapsto (gx,x)$. Let $G$ act diagonally on $M\times M$.
Then $\phi(G\times M)$ is a closed, smooth  $G$-invariant submanifold of $M\times M$, and the map
$$
f\colon G/H\times M\to M\times M,\,\,\, (gH,x)\mapsto (gx,x),
$$
is a smooth embedding with image $\phi(G\times M)$.
\end{lemma}

\begin{proof}
The map $\phi$ is a proper map, since $G$ acts properly on $M$. Thus $\phi$ is a closed map, which implies that
$\phi(G\times M)$ is closed in $M\times M$. Clearly, $\phi(G\times M)$ is $G$-invariant.

The map $f$ is well-defined, because $G_x=H$, for every $x\in M$. Clearly, the image of $f$ is
$\phi(G\times M)$. The map $f$ is injective, the inverse map, defined in the image of $f$ is
$$
f^{-1}\colon \phi(G\times M)\to G/H\times M,\,\,\, (gx,x)\mapsto (gH,x).
$$
Let $(gx,x)\in M\times M$, and let ${\rm N}$ be a linear slice at $x$.
Then $G{\rm N}$ is open in $M$, and the map
$$
\nu\colon G{\rm N}\times G{\rm N}\to G/H\times G{\rm N},\,\,\, (gy,z)\mapsto (gH,z),
$$
where $g\in G$ and $y\in{\rm N}$,
is smooth. The maps $f^{-1}$ and $\nu$ agree in the intersection 
of their domains. 
Thus the restriction of $f^{-1}$ to this intersection is continuous. Since $x$
was chosen arbitrarily, it follows that $f^{-1}$ is continuous.

Let $g_0H\in  G/H$. Then $g_0H$ has an open neighborhood
$U$ such that the quotient map  $G\to G/H$, $g\mapsto gH$, has a smooth local section $\gamma\colon U\to G$. 
Then $f(gH,x)=(\gamma(gH)x,x)$, for every $(gH,x)\in U\times M$. It follows that the restriction of $f$ to $U\times M$ is smooth and
furthermore that $f$ is smooth. To complete the proof, it remains to show that $f$ is an immersion.  Let
$(gH,x)\in G/H\times M$. The restriction of $f$ to $G/H\times \{x\}$ is a smooth embedding with the image $Gx\times \{ x\}$.
Thus the restriction of $df$ to ${\rm T}_{gH}(G/H)$ is injective. Moreover, the restriction of $f$ to $\{ gH\}\times M$ is a smooth 
embedding, and thus also the restriction of $df$ to ${\rm T}_xM$ is injective. Since
$$
df\bigl({\rm T}_{gH}(G/H)\bigr)\cap df({\rm T}_xM)=\{0\},
$$
it follows that $df$ is injective. Thus $f$ is an immersion.
\end{proof}

\begin{lemma}
\label{avoinkuvaus}
Let $X$ and $Y$ be topological spaces, and let $f\colon X\to Y$ be an open map. Assume 
$Z\subset Y$ satisfies $f\bigl(f^{-1}(Z)\bigr)=Z$. Then the induced map
$$
\tilde{f}\colon f^{-1}(Z)\to Z,\,\,\, x\mapsto f(x),
$$
is open. 
\end{lemma}

\begin{proof}
Let $U$ be an open subset of $f^{-1}(Z)$. Then $U=O\cap f^{-1}(Z)$ for some open subset
$O$ of $X$, and 
$\tilde{f}(U)=f\bigl(O\cap f^{-1}(Z)\bigl)$. We will show that $f\bigl(O\cap f^{-1}(Z)\bigr)=f(O)\cap Z$, which will
prove the claim. Clearly, 
$$
f\bigl(O\cap f^{-1}(Z)\bigr)\subset f(O)\cap f\bigl(f^{-1}(Z)\bigr)=f(O)\cap Z.
$$
Assume then $y\in f(O)\cap Z$. Then $y=f(x)$, for some $x\in O$. But then
$x\in f^{-1}(y)\subset f^{-1}(Z)$, which implies that $x\in O\cap f^{-1}(Z)$. Therefore,
$y=f(x)\in f\bigl(O\cap f^{-1}(Z)\bigr)$. Consequently, $f(O)\cap Z\subset f\bigl(O\cap f^{-1}(Z)\bigr)$.
\end{proof}

Let $H$ be a compact subgroup of a Lie group $G$. We denote the {\it normalizer}
of $H$ in $G$ by $N(H)$ and the connected component of $N(H)$ containing the identity
element by $N(H)_0$. Then $N(H)$ and $N(H)_0$ are closed subgroups of $G$.

\begin{lemma}
\label{MHconncomp}
Let $G$ be a Lie group and let $H$ be a compact subgroup of $G$. Let $M$ be a proper smooth $G$-manifold,
and let $\pi\colon M\to M/G$ be the natural projection. Then the connected components of $\pi(M_H)$ are the sets
$\pi(Y_i)$, where $Y_i$, $i\in J$, are the connected components of $M_H$.
\end{lemma}

\begin{proof}
Since the sets $Y_i$ are connected, also the images $\pi(Y_i)$ are connected. Assume  $\pi(Y_i)\cap \pi(Y_j)\not=\emptyset$,
for some $i,j\in J$. Then there are $g_i, g_j\in G$, $y_i\in Y_i$ and $y_j\in Y_j$, with $g_iy_i=g_jy_j$. Hence
$g_j^{-1}g_iy_i=y_j\in M_H$, and it follows that $g_j^{-1}g_i\in N(H)$. As  a diffeomorphism
$M_H\to M_H$,  $g_j^{-1}g_i$ takes connected components onto connected components. 
Thus $g_j^{-1}g_iY_i=Y_j$, and $\pi(Y_i)=\pi(Y_j)$. 
By Lemma \ref{avoinkuvaus}, the restriction $\pi\vert \colon M_{(H)}\to \pi(M_{(H)})$ is open.
It follows from Theorem 4.3.10 in \cite{P}, that for any open subset $A$ in  $M_H$, the subset $GA$ is open
in $M_{(H)}$. Thus also the restriction $\pi\vert\colon  M_H\to \pi(M_H)$ is open.
Since $M_H$ is a $\Sigma$-manifold, each $Y_i$ is open in $M_H$. 
Thus each $\pi(Y_i)$ is open in $\pi(M_H)$, and the claim follows.
\end{proof}

\section{Equivariant isotopies}
\label{isot.section}

\noindent Let $G$ be a Lie group and let $M$ and $N$ be  proper smooth $G$-manifolds.  
Let $I=[0,1]$. A map
$$
F\colon M\times I\to N
$$
is a $G$-{\it equivariant isotopy}, if it is smooth and  if every map
$$
F_t\colon M\to N,\,\,\, x\mapsto F(x,t),
$$
is a smooth $G$-equivariant diffeomorphism.
The {\it support}  ${\rm supp}(f)$ of
a $G$-equivariant diffeomorphism $f\colon M\to M$ is the closure of the set
$$
\{ x\in M\mid f(x)\not=x\},
$$
and the support of a $G$-equivariant isotopy $F\colon M\times I\to M$ is the closure of the set
$$
\{ x\in M\mid F(x,t)\not= x\,\, {\rm for \,\, some \,\,} t\in I\}.
$$
The support of a map $f\colon M\to {\mathbb{R}}$ is the closure of the set
$\{ x\in M\mid f(x)\not=0\}$.
An equivariant diffeomorphism, isotopy  or a real valued map
is called $G$-{\it compactly supported}, if it's support is
$G$-compact. We denote by ${\rm Diff}^G_c(M)$
the group of $G$-equivariant diffeomorphisms 
of $M$ isotopic to the identity through a $G$-equivariant $G$-compactly supported isotopy.

\begin{theorem}
\label{uusi.isot.jatko}
Let $G$ be a Lie group and let $M$ be a proper smooth $G$-manifold. Let $N$ be a closed smooth $G$-invariant submanifold
of $M$. Let $F\colon N\times I\to M$ be a $G$-equivariant,
$G$-compactly supported  isotopy of embeddings
such that $F_0$ is the canonical inclusion. 
Then $F$ extends to a $G$-equivariant, $G$-compactly supported
isotopy of $M$ starting at the identity.
\end{theorem}

\begin{proof}
The proof is similar to the proof of Theorem 8.6 in \cite{K}. When the assumption of the isotopy having
bounded velocity is replaced by the assumption of the isotopy  being $G$-compactly supported, also the obtained
isotopy can be made $G$-compactly supported. Notice that what we call an isotopy here is called a diffeotopy in 
\cite{K} and what we call an isotopy of embeddings is called an isotopy in \cite{K}.
\end{proof}

\begin{lemma}
\label{viip.isot.}
Let $G$ be a Lie group and let $M$ be a proper smooth $G$-manifold. Let $x\in M$ and let ${\rm N}$ be a linear slice at $x$. 
Denote the isotropy subgroup $G_x$ of $x$ by $H$.  Assume $y\in N(H)_0{\rm N}_H$.
Then there exists a smooth $G$-equivariant isotopy 
$$
F\colon G{\rm N}_H\times I\to  G{\rm N}_H,
$$
with $G$-compact support, starting at the identity and such that $F_1(x)=y$.
\end{lemma}

\begin{proof}  Assume first that $y\in {\rm N}_H$.
Since ${\rm N}$ is a linear slice at $x$, we may identify it with an orthogonal $H$-space. Then ${\rm N}_H$ is the fixed point set
of $H$ in ${\rm N}$, and hence connected as a linear subspace of ${\rm N}$.  Thus the group of compactly supported diffeomorphisms 
isotopic to the identity of ${\rm N}_H$ acts transitively on 
${\rm N}_H$ (see \cite{MV}). Notice that this holds also if  $\dim({\rm N}_H)=1$. 
Let $f\colon {\rm N}_H\to {\rm N}_H$ be such a diffeomorphism taking $x$ to $y$, and let $F$ be a
compactly supported  isotopy from ${\rm id}_{{\rm N}_H}$ to $f$.
Since $H$ acts trivially on ${\rm N}_H$, it follows that
the diffeomorphisms $F_t$, $t\in I$, are $H$-equivariant. Then
$$
\tilde{F}\colon G{\rm N}_H\times I\to G{\rm N}_H,\,\,\, (gz,t)\mapsto gF(z,t),
$$
is a $G$-equivariant,
$G$-compactly supported isotopy starting at the identity and such that $\tilde{F}_1(x)=y$.

Let then $k\in N(H)_0$. Let $u\colon I\to N(H)_0$ be a smooth path from the identity element $e$ of $G$ to
$k$. Let $\psi\colon {\rm N}_H\to I$ be a smooth  function having compact support and taking
$y$ to $1$. Let
$$
\varphi\colon {\rm N}_H\times I\to G{\rm N}_H,\,\,\, (z,t)\mapsto u\bigl(\psi(z)t\bigr)z.
$$
Then $\varphi_0$ is the inclusion ${\rm N}_H\hookrightarrow G{\rm N}_H$, $\varphi_1(y)=ky$, and each $\varphi_t$
is a smooth $H$-equivariant map. It is easy to check that each $\varphi_t$ is injective. 
We check that each
$\varphi_t$ is an immersion: Let
$$
f_0\colon G{\rm N}\to G/H,\,\,\, gz\mapsto gH,
$$
for all $g\in G$ and for all $z\in {\rm N}$.
Since ${\rm N}$ is a linear slice at $x$, there is a smooth cross-section $\sigma\colon O\to G$ of the
map $G\to G/H$, $g\mapsto gH$, defined in some $H$-invariant open neighborhood $O$ of $eH$ in
$G/H$, and a diffeomorphism
$$
q\colon O\times {\rm N}\to W,\,\,\, (o,z)\mapsto \sigma(o)z,
$$
onto an open subset $W$ of $M$. The inverse map of $q$ is 
$$
q^{-1}\colon W\to O\times {\rm N}, \,\,\, z\mapsto \Bigl( f_0(z), \sigma\bigl( f_0(z)\bigr)^{-1}z\Bigr).
$$
Similarly, for every $g_0\in G$, there is a diffeomorphism
$$
q_{g_0}^{-1}\colon g_0W\to g_0O\times {\rm N}, \,\,\,
z\mapsto \Bigl( g_0f_0(g_0^{-1}z), \sigma\bigl( f_0(g_0^{-1}z)\bigr)^{-1}g_0^{-1}z\Bigr).
$$
Let $z\in {\rm N}_H$. Then $\varphi_t(z)\in g_0W$, for some $g_0\in G$. Let
$$
{\rm pr}\colon g_0O\times {\rm N}\to {\rm N}
$$
be the projection. Then
\begin{align*}
({\rm pr}\circ q_{g_0}^{-1}\circ\varphi_t)(z)&= ({\rm pr}\circ q_{g_0}^{-1})
\Bigl( u\bigl(\psi(z)t\bigr)z\Bigr)\\
&=\sigma\Bigl( f_0\bigl( g_0^{-1}u\bigl(\psi(z)t\bigr)z\bigr)\Bigr)^{-1}
g_0^{-1}u\bigl(\psi(z)t\bigr)z\\
&=hz\\
&=z,
\end{align*}
for some $h\in H$. Thus the restriction
$$
{\rm pr}\circ q_{g_0}^{-1}\circ \varphi_t\vert \colon \varphi^{-1}_t(g_0W)\to {\rm N}
$$
is the inclusion, which implies that $\varphi_t$ is immersive at $z$.

We next show that
each $\varphi_t$ is a closed map: Let $A$ be a closed subset of ${\rm N}_H$. Assume $w$ is a point in the closure of
the image $\varphi_t(A)$. Then there is a sequence of points $\varphi_t(z_n)=u\bigl(\psi(z_n)t\bigr)z_n$ converging to $w$, where
$z_n\in A$, for every $n$.
The set
$$
K=\bigl\{ u\bigl( \psi(z)t\bigr)\mid z\in {\rm N}_H\bigr\}
$$
is compact, since $\psi({\rm N}_H)=I$  and $u$ is a path. By passing to a subsequence, if necessary, we may assume that
the sequence of points $u\bigl( \psi(z_n)t\bigr)$ converges to a point $u\bigl( \psi(z)t\bigr)$, for some $z\in {\rm N}_H$.
It follows that the sequence $(z_n)$ converges to $u\bigl( \psi(z)t\bigr)^{-1}w=\tilde{z}$. 
Since $A$ is closed, it follows that
$\tilde{z}\in A$. The sequence $\Bigl(u\bigl(\psi(z_n)t\bigr)z_n\Bigr)$ converges to $u(\psi(z)t)\tilde{z}$. 
Now, $z_n\to \tilde{z}$ implies $u\bigl(\psi(z_n)t\bigr)\to u\bigl(\psi(\tilde{z})t\bigr)$. Thus
$u\bigl( \psi(\tilde{z})t)\bigr)=u\bigl( \psi(z)t\bigr)$ and hence
$w=u\bigl(\psi(z)t\bigr)\tilde{z}=u\bigl(\psi(\tilde{z})t\bigr)\tilde{z}=
\varphi_t(\tilde{z})\in \varphi_t(A)$.
It follows that $\varphi_t(A)$ is closed, and hence that $\varphi_t$ is a closed
map.

We proved that $\varphi$ is a compactly supported $H$-equivariant isotopy of embeddings. Let
$$
\tilde{\varphi}\colon G{\rm N}_H\times I \to G{\rm N}_H,\,\,\, (gz,t)\mapsto g\varphi(z,t).
$$
Clearly, $\tilde{\varphi}_1(y)=ky$.
To show that
$\tilde{\varphi}$ is a $G$-compactly supported $G$-equivariant isotopy starting at the identity, it
suffices to show that each $\tilde{\varphi}_t$ is bijective, and that the inverse map
$\tilde{\varphi}_t^{-1}$ is smooth.
Let $g_0\in G$ and $z_0\in {\rm N}_H$, so that $g_0z_0\in G{\rm N}_H$.
Then $g_0u\bigl(\psi(z_0)t\bigr)^{-1}\in G$ and 
$$
\tilde{\varphi}_t\Bigl(g_0u\bigl(\psi(z_0)t\bigr)^{-1}z_0\Bigr)=g_0u\bigl(\psi(z_0)t\bigr)^{-1}u\bigl(\psi(z_0)t\bigr)z_0=g_0z_0.
$$ Thus $\tilde{\varphi}_t$ is surjective.

Assume then $\tilde\varphi_t(g_1z_1)=\tilde\varphi_t(g_2z_2)$. Then $g_1u\bigl(\psi(z_1)t\bigr)z_1=
g_2u\bigl(\psi(z_2)t\bigr)z_2$. Thus $g_1u\bigl(\psi(z_1)t\bigr)=g_2u\bigl(\psi(z_2)t\bigr)h$, for some 
$h\in H$. It follows that $z_1=z_2$, and moreover that $u\bigl(\psi(z_1)t)\bigr)=u\bigl(\psi(z_2)t\bigr)$.
Since $u\bigl( \psi(z_2)t\bigr)^{-1}g_1^{-1}g_2u\bigl(\psi(z_2)t\bigr)\in H$, 
$g_1^{-1}g_2\in u\bigl( \psi(z_2)t\bigr) Hu\bigl( \psi(z_2)t\bigr)^{-1}=H$. Thus $g_1=g_2h'$, for some
$h'\in H$, and it follows that $g_1z_1=g_2z_2$. Hence $\tilde\varphi_t$ is injective.

It remains to check that $\tilde{\varphi}_t^{-1}\colon G{\rm N}_H\to G{\rm N}_H$ is smooth. An arbitrary element in 
$G{\rm N}_H$ can be written as $g_0u\bigl( \psi(z_0t)\bigr)z_0$, where $g_0\in G$ and $z_0\in
{\rm N}_H$. Then $g_0u\bigl(\psi(z_0)t\bigr)\in G$. The restriction
$$
f_0\vert \colon G{\rm N}_H\to G/H,\,\,\, gz\mapsto gH,
$$
is smooth. The map
$G\to G/H$, $g\mapsto gH$,
has a smooth local cross-section $\sigma\colon U\to G$, where $U$ is an open
$H$-invariant neighborhood of $g_0u\bigl(\psi(z_0)t\bigr)H$ in $G/H$. The set
$f_0^{-1}(U)$ is open in $G{\rm N}_H$ and $g_0u\bigl(\psi(z_0)t\bigr)z_0\in f_0^{-1}(U)$.
The map
\begin{align*}
(\sigma^{-1}\circ f_0\vert, {\rm id})\colon   f_0^{-1}(U)& \to G\times f_0^{-1}(U), \\
 gu\bigl(\psi(z)t\bigr)z &\mapsto 
\Bigl( \sigma\bigl(gu\bigl(\psi(z)t\bigr)H\bigr)^{-1}, gu\bigl(\psi(z)t\bigr)z\Bigr), 
\end{align*}
is smooth. Since $G$ acts smoothly on $M$, it follows that the map
$$
f_1\colon f_0^{-1}(U)\to G{\rm N}_H,\,\,\,
gu\bigl( \psi(z)t\bigr)z\mapsto
\sigma\bigl( gu\bigl(\psi(z)t\bigr)H\bigr)^{-1}gu\bigl( \psi(z)t\bigr)z=z,
$$
is smooth. Finally, the map
$$
p\colon {\rm N}_H\to G{\rm N}_H,\,\,\, 
z\mapsto u\bigl(\psi(z)t\bigr)^{-1}z,
$$
is smooth. Thus the map
\begin{align*}
(\sigma\circ f_0\vert, p\circ f_1)\colon 
f_0^{-1}(U)&\to G\times G{\rm N}_H,\\
 gu\bigl(\psi(z)t\bigr)z&\mapsto 
\Bigl( \sigma\bigl(gu\bigl(\psi(z)t\bigr)H\bigr), 
u\bigl( \psi(z)t\bigr)^{-1}z\Bigr),
\end{align*}
is smooth. Consequently,
\begin{align*}
\tilde{\varphi}_t^{-1}\vert\colon   f_0^{-1}(U)&\to G{\rm N}_H,\\
  gu\bigl(\psi(z)t\bigr)z&\mapsto
 \sigma\bigl( gu\bigl(\psi(z)t)H\bigr)u\bigl(\psi(z)t\bigr)^{-1}z=gz,
\end{align*}
is smooth.

An isotopy, as in the claim of the lemma, is obtained as a composition of isotopies of the form $\tilde{F}$ and
$\tilde{\varphi}$ of the proof.
\end{proof}

The following proposition will be used to prove Theorem \ref{result1}.

\begin{prop}
\label{apulause}
Let $G$ be a Lie group and let $H$ be a compact subgroup of $G$. Let $M$ be a proper smooth $G$-manifold and
let $x\in M$  with $G_x=H$. Let $U$ be an open $G$-invariant
neighborhood of $x$. Then there exists an open  neighborhood $V$ of $x$ in $M_{H}$ with the property that $GV\subset U$,
and for each $y\in V$ there exists $f\in {\rm Diff}^G_c(M)$ such that $f(x)=y$ and $f$ is isotopic to the identity through a
$G$-equivariant isotopy with $G$-compact support contained in $U$. 
\end{prop}

\begin{proof}
We may assume that $U=G{\rm N}$, for some linear slice ${\rm N}$ at $x$.
Let $V=N(H)_0{\rm N}_H$. Since $G{\rm N}$ is open in $M$, it follows that
$N(H){\rm N}_H=G{\rm N}\cap M_H$ is open in $M_H$.  The map
$$
k\colon N(H){\rm N}_H\to N(H)/H,\quad gz\mapsto gH,
$$
is continuous as a restriction of the continuous map
$G{\rm N}\to G/H$, $gz\mapsto gH$. Thus
$V=N(H)_0{\rm N}_H= k^{-1}\bigl( N(H)_0H/H\bigr)$ is open 
in $N(H){\rm N}_H$ and, consequently, also in $M_H$.
Assume $y\in V$.
By Lemma \ref{viip.isot.}, there is a
smooth $G$-equivariant isotopy 
$$
F\colon G{\rm N}_H\times I\to  G{\rm N}_H,
$$
with $G$-compact support,   starting at the identity and such that $F_1(x)=y$.
Now, ${\rm N}_H$ is closed in ${\rm N}$. Thus $G{\rm N}\setminus G{\rm N}_H=
G({\rm N}\setminus {\rm N}_H)$ is open in $G{\rm N}$, and hence $G{\rm N}_H$
is closed in $G{\rm N}$. By Theorem \ref{uusi.isot.jatko}, the isotopy $F$ can be
extended to a smooth
$G$-equivariant isotopy 
$$
\tilde{F}\colon G{\rm N}\times I\to  G{\rm N},
$$
with $G$-compact support and starting at the identity.
Since $G{\rm N}$ is open in $M$, and since the support of $\tilde{F}$ is $G$-compact, it follows that
extending each $\tilde{F}_t$, $t\in I$, to be the identity on $M\setminus G{\rm N}$,
yields an isotopy $M\times I\to M$ having the properties in the claim.
\end{proof}

\section{Orbifolds}
\label{orbi}

\noindent  In this section we recall the definition and some basic properties of orbifolds.

\begin{definition}
\label{ensin}
Let $X$ be a topological space and let $n> 0$.
{\begin{enumerate}
\item An $n$-dimensional {\it orbifold chart} for an open subset $V$
of  $X$ is a triple
$(\tilde{V}, G,\varphi)$ satisfying the following conditions:
{\begin{enumerate}
\item $\tilde{V}$ is a connected open subset of ${\mathbb{R}}^n$,

\item $G$ is a finite group of homeomorphisms acting  on $\tilde{V}$, 
${\rm ker}(G)$ denotes the subgroup of $G$ acting trivially on $\tilde{V}$,

\item $\varphi\colon \tilde{V}\to V$ is a $G$-invariant map that induces a
homeomorphism from $\tilde{V}/G$ onto  $V$.
\end{enumerate}}

\item  If $V_i\subset V_j$, an {\it embedding} $(\lambda_{ij}, h_{ij})
\colon (\tilde{V}_i, G_i, \varphi_i)\to
(\tilde{V}_j, G_j,\varphi_j)$ means 
{\begin{enumerate}
\item an injective homomorphism $h_{ij}\colon G_i\to G_j$, 
such that $h_{ij}$ is an isomorphism from ${\rm ker}(G_i)$ to
${\rm ker}(G_j)$,
and
\item an equivariant embedding
$\lambda_{ij}\colon \tilde{V}_i\to \tilde{V}_j$ such that $\varphi_j\circ
\lambda_{ij}=\varphi_i$. (Thus 
$\lambda_{ij}(gx)=h_{ij}(g)\lambda_{ij}(x)$ for every $g\in G_i$
and every $x\in \tilde{V}_i$.)
\end{enumerate}}

\item  An {\it orbifold atlas} on $X$ is a family ${\mathcal V}=\{ (\tilde{V_i}, G_i, \varphi_i)\}_{
i\in J}$ of orbifold charts satisfying the following conditions:
{\begin{enumerate}
\item $\{ V_i\}_{i\in J}$ is a covering of $X$,

\item  given two charts $(\tilde{V}_i, G_i, \varphi_i)$ and
$(\tilde{V}_j, G_j, \varphi_j)$ and a point $x\in V_i\cap V_j$, there exists
an open neighborhood $V_k\subset V_i\cap V_j$ of $x$ and a chart
$(\tilde{V}_k, G_k, \varphi_k)$ such that there are embeddings 
$(\lambda_{ki}, h_{ki})\colon (\tilde{V}_k,G_k, \varphi_k)\to
(\tilde{V}_i, G_i,\varphi_i)$ and
$(\lambda_{kj}, h_{kj})\colon (\tilde{V}_k,G_k, \varphi_k)\to
(\tilde{V}_j, G_j,\varphi_j)$.
\end{enumerate}}

\item An atlas ${\mathcal U}$ is called a {\it refinement} of an atlas
${\mathcal W}$ if  every chart in ${\mathcal U}$ admits an embedding
into some chart of ${\mathcal W}$. Two orbifold atlases 
having a common refinement are called
{\it equivalent}. 
\end{enumerate}}
\end{definition}
 
 \begin{definition}
 An $n$-dimensional {\it orbifold} is a paracompact Hausdorff space $X$
 equipped with an equivalence class of $n$-dimensional orbifold atlases.
 \end{definition}

If the actions of the finite groups on the orbifold charts are effective, the
orbifold is called {\it reduced}. Reduced orbifolds are also called {\it effective}.
  
 An orbifold is  called {\it smooth}
 if each $G_i$ acts by  smooth diffeomorphisms  on 
 $\tilde{V}_i$ and if each embedding $\lambda_{ij}\colon {\tilde{V}}_i\to
 \tilde{V}_j$ is smooth. By the differentiable slice theorem we may always choose
 orbifold charts $(\tilde{V_i}, G_i, \varphi_i)$  of a smooth orbifold
 in such a way that $\tilde{V}_i$ is diffeomorphic to
 a euclidean space ${\mathbb{R}}^n$, where $n$ is the dimension of the orbifold, and
 $G_i$ acts linearly on ${\mathbb{R}}^n$. In this paper we only consider smooth orbifolds.

Let $G$ be a compact Lie group acting smoothly,  effectively and almost freely on a smooth manifold $M$. Then the quotient
space $M/G$ is a reduced smooth orbifold. 
The orbifold charts of $M/G$ are the triples $({\rm N}_x, G_x, \pi_x)$, where 
${\rm N}_x$ is a linear slice at $x\in M$ and
$\pi_x\colon {\rm N}_x\to {\rm N}_x/G_x\cong (G{\rm N}_x)/G$ denotes the natural projection.  

The {\it frame bundle} construction allows us to consider any reduced smooth orbifold as a quotient space of 
an action by a compact Lie group. Namely, by Theorem 1.23 in \cite{ALR}, there is the following result:
Let $X$ be a reduced smooth orbifold of dimension $n$. Then the frame bundle ${\rm Fr}(X)$ is a smooth manifold with
a smooth, effective and almost free action by the orthogonal group ${\rm O}(n)$. The orbifold $X$ is naturally
isomorphic to the resulting quotient orbifold ${\rm Fr}(X)/{\rm O}(n)$.

\begin{definition}
\label{orbimap}
Let $X$ and $Y$ be smooth orbifolds. A map
$f\colon X\to Y$ is called a {\it  smooth orbifold map}, 
if for every $x\in X$, there
are charts $(\tilde{U}, G,\varphi)$ around $x$ and $(\tilde{V}, H,\psi)$
around $f(x)$, such that $f$ maps $U=\varphi(\tilde{U})$ into
$V=\psi(\tilde{V})$ and the restriction $f\vert U$
can be lifted to a smooth
equivariant map
$\tilde{f}\colon\tilde{U}\to \tilde{V}$.
A  smooth orbifold map
$f\colon X\to Y$ is called an {\it orbifold diffeomorphism}, if there is a
smooth orbifold map $g\colon Y\to X$ such that $g\circ f={\rm id}_X$ and
$f\circ g={\rm id}_Y$.
\end{definition}

The {\it support} ${\rm supp}(f)$ of an orbifold diffeomorphism $f\colon X\to X$ is the closure  of the set
$$
\{ x\in X\mid f(x)\not= x\}.
$$
We give an orbifold structure to $X\times I$ as follows: 
If $(\tilde{U}_i, G_i, \varphi_i)$ is an orbifold chart of $X$, then $(\tilde{U}_i\times I, G_i, \varphi_i\times {\rm id}_I)$
is an orbifold chart of $X\times I$, where 
$G_i$ acts trivially on $I$ and diagonally on $\tilde{U}_i\times I$.
This kind of charts define an orbifold atlas for $X\times I$.

Assume then $X$ and $Y$ are smooth orbifolds and assume 
$f_0, f_1\colon X\to  Y$ are  two smooth orbifold maps. If there is a smooth orbifold
map
$$
F\colon X\times I\to Y,
$$
with $F_0=f_0$ and $F_1=f_1$, we say that $f_0$ and $f_1$ are smoothly homotopic. Here, for every
$t\in I$, $F_t\colon X\to Y$, $x\mapsto F(x,t)$. If, in addition,  $F_t$ is a smooth orbifold
diffeomorphism for every $t\in I$, we say that $f_0$ and $f_1$ are {\it smoothly isotopic}.

\section{Orbifold stratifications}
\label{strata}

\noindent Orbifolds admit stratifications  by {\it singular dimension}. This stratification is explained in detail in
\cite{D} and used in \cite{PR}: Let $X$ be a reduced smooth
$n$-dimensional orbifold, and let
$x\in X$. Let $(\tilde{V}_x, G_x, \varphi_x)$ be an  orbifold chart centered at  $x$, and let $V=\varphi_x(\tilde{V}_x)$. There is a
unique point $\tilde{x}\in \tilde{V}_x$ with $\varphi_x(\tilde{x})=x$. 
 If $(\tilde{V}'_x, G'_x, \varphi'_x)$ is another orbifold chart centered at $x$, then the groups
$G_x$ and $G'_x$ are isomorphic. The action of $G_x$ on $\tilde{V}_x$ fixes $\tilde{x}$ and the 
differential induces a linear action on
the tangent space ${\rm T}_{\tilde{x}}\tilde{V}_x$.  
Denote by  ${\rm T}_{\tilde{x}}\tilde{V}_x^{G_x}$ the subspace of ${\rm T}_{\tilde{x}}\tilde{V}_x$
consisting of the fixed points of the $G_x$-action. The
singular dimension of $x$ is then defined to be ${\rm sdim}(x)={\rm dim}({\rm T}_{\tilde{x}}\tilde{V}_x^{G_x})$, and it does not depend on the choice of the orbifold chart.
The singular set
$$
\Sigma=\{ x\in X\mid G_x\not=1\}
$$
of $X$  can be written as the union  of the {\it singular strata} $\Sigma_k$, $0\leq k\leq n-1$, where
$$
\Sigma_k=\{ x\in \Sigma\mid {\rm sdim}(x)=k\}.
$$
Similarly, for a point $x\in X$ with trivial $G_x$, the singular dimension ${\rm sdim}(x)=n$, and $\Sigma_n$ is the union of such points. 
Then $X=\bigcup_{i=0}^n\Sigma_i$.
The connected component of $\Sigma_{{\rm sdim}(x)}$ containing $x$ is 
a smooth manifold denoted by $\Sigma(x)$. The isomorphism class of the 
isotropy group $G_y$ is constant for $y\in \Sigma(x)$.  

For a finite group $H$, let 
$$
X_H=\{ x\in X\mid G_x\cong H\}.
$$
For $x\in X$, denote by $\Sigma'(x)$ the connected component of the subset $X_{G_x}$ of $X$ containing $x$.
Then $\Sigma(x)\subset \Sigma'(x)$.  
Cover $\Sigma'(x)$ by open sets $V_i$, where
$(\tilde{V}_i, G_i, \varphi_i)$ is a chart for  $V_i$ centered at some point in $\Sigma'(x)$,
$\tilde{V}_i$ is homeomorphic to ${\mathbb{R}}^n$ and $n$ is the dimension of $X$. 
The following lemma then implies that also $\Sigma'(x)\subset \Sigma(x)$, and hence
that $\Sigma(x)=\Sigma'(x)$.

\begin{lemma}
\label{samatsigmat}
Let $H$ be a finite group acting smoothly on a euclidean space ${\mathbb{R}}^n$. 
Let $F$ be
a connected component of the fixed point set of the action. Assume $x,y\in F$. 
The group $H$
acts on the tangent spaces ${\rm T}_x{\mathbb{R}}^n$ and ${\rm T}_y{\mathbb{R}}^n$ via the 
differentials.
Then $\dim({\rm T}_x{\mathbb{R}}^n)^H=\dim({\rm T}_y{\mathbb{R}}^n)^H$.
\end{lemma}

\begin{proof}
Notice that the fixed point set ${\mathbb{R}}^n_H$ is a $\Sigma$-manifold. Thus the connected
components of ${\mathbb{R}}^n_H$ equal the path components.
Denote the set of group homomorphisms from $H$ to 
the general linear group ${\rm GL}_n({\mathbb{R}})$,
equipped with the compact-open topology,
by 
$$
{\rm hom}\bigl( H,{\rm GL}_n({\mathbb{R}})\bigr).
$$
Let 
$f\colon I\to F$ be a smooth path from $x$ to $y$. The map
$$
G\colon H\times I\to {\rm GL}_n({\mathbb{R}}),\,\,\, (h,t)\mapsto dh_{f(t)},
$$
is continuous.  Since $H$ is compact, $G$  induces a continuous map
$$
G^\#\colon I\to {\rm hom}\bigl( H,{\rm GL}_n({\mathbb{R}})\bigr),\,\,\, t\mapsto G_t,
$$
where $G_t(h)=G(h,t)$.
By Lemma VIII 38.1 in \cite{CF}, any group homomorphism
$H\to {\rm GL}_n({\mathbb{R}})$ that is sufficiently close to
$G_t$ must be conjugate to $G_t$. Since $I$ is compact and $G^\#$ is continuous, the homomorphisms
$G_0$ and $G_1$ are conjugate. The claim follows, since conjugate actions have isomorphic 
fixed point spaces.
\end{proof}

By using the frame bundle construction, we know that a reduced orbifold  $X$ may be considered as a quotient space $M/G$, where
$G$ is a compact Lie group acting smoothly, effectively and almost freely on a smooth manifold $M$. Then  $M$ has a natural stratification by
orbit types $M_{(H)}$, where $H$ is a finite subgroup of $G$ (see \cite{P}, Theorem 4.3.7).
If $M$ is compact, then the action has only finitely many orbit types. 
We  denote by $\tilde\Sigma(x)$ the connected component of $M_{(G_x)}$ containing $x$. This stratification of $M$ induces a stratification on the
orbit space $M/G$: If $\pi\colon M\to M/G$ is the natural projection, then the strata on $M/G$ are the images $\pi(M_{(H)})$ of the strata of $M$
(\cite{P}, 4.3.9).
Denote by $\hat{\Sigma}\bigl(\pi(x)\bigr)$ the connected component of  the stratum containing the point $\pi(x)\in M/G$, where $x\in M$. 
Let ${\rm N}$ be a linear slice at $z$, where $\pi(z)\in \Sigma'\bigl( \pi(x)\bigr)$.
Then $G{\rm N}$ is open in $M$ and for any $y\in {\rm N}$, $G_y\subset G_z$. Thus 
$\{y\in M\mid G_y\cong G_z\cong G_x\}\cap G{\rm N}\subset M_{(G_x)}$. 
Now $\pi(G{\rm N})$ is open in $M/G$, and
$\Sigma'\bigl(\pi(x)\bigr)\cap\pi(G{\rm N})=\hat{\Sigma}\bigl(\pi(x)\bigr)\cap\pi(G{\rm N})$. By covering $M/G$ by sets of the form
$\pi(G{\rm N})$, we see that in fact $\Sigma'\bigl(\pi(x)\bigr)=\hat{\Sigma}\bigl(\pi(x)\bigr)$. Thus, 
although the concept of conjugacy of the local groups of an orbifold does
not make sense, it follows that when an orbifold is considered as an orbit space, the local groups of points on a connected component of a 
stratum $\Sigma'$ are not just isomorphic but actually conjugate.

\section{Smooth maps between orbit spaces}
\label{smoothmaps}

\noindent  In \cite{S}, G. Schwarz considered smooth isotopies between orbit spaces. 
In this section we look at how his concept of smoothness is
related to smoothness in the orbifold sense.
Thus, let $G$ be a compact Lie group and let $M$ be a smooth $G$-manifold. Let $\pi\colon M\to M/G$ denote the natural projection.
We give the orbit space $M/G$ the quotient space topology and
differentiable structure as in \cite{S}: If $U$ is an open subset of $M/G$, then ${\rm C}^\infty(U)$ is the set of functions $f\colon U\to {\mathbb{R}}$
for which the composition $f\circ \pi\vert\colon \pi^{-1}(U)\to {\mathbb{R}}$ is smooth. Assume $N$ is another smooth $G$-manifold. Then a map
$\psi\colon M/G\to N/G$ is defined to be smooth, if $f\circ \psi\colon M/G\to {\mathbb{R}}$ 
is smooth for every smooth $f\colon N/G\to {\mathbb{R}}$. 
In particular, a map $M/G\to N/G$ is smooth, if it is  induced by a smooth $G$-equivariant map $M\to N$.
The notions of diffeomorphism
and isotopy have their usual categorical meaning. In \cite{S}, Schwarz proves the following isotopy lifting conjecture (\cite{S}, p. 38):

\begin{theorem}
\label{isot.lift}
Let $G$ be a compact Lie group, and let $M$ be a smooth $G$-manifold. Let $\bar{F}\colon M/G\times I\to M/G$ be a smooth isotopy starting at the identity.
Then there is a smooth $G$-equivariant isotopy $F\colon M\times I\to M$ starting at the identity and inducing $\bar{F}$.
\end{theorem}

Let  $G$ be  a compact Lie group and let $M$ and $N$ be   smooth $G$-manifolds.
Assume the actions  are effective  and almost free so that the  quotient spaces
$M/G$  and $N/G$  are smooth orbifolds.
Let $\pi\colon M\to M/G$   and $\pi'\colon N\to N/G$ denote the natural projections.
Assume $\psi\colon M/G\to N/G$ is a smooth orbifold map.
Let $x\in M/G$. Then there 
are charts $({\rm N}_x, G_x, \pi_x)$ around $x$ and $({\rm N}_y, G_y, \pi'_y)$
around $y=\psi(x)$, such that $\psi$ maps 
$U={\rm N}_x/G_x$ into $V={\rm N}_y/G_y$
and the restriction $\psi\vert U$
can be lifted to a smooth
equivariant map $\tilde{\psi}\colon {\rm N}_x\to {\rm N}_y$.
The sets $\pi^{-1}(U)=G\times_{G_x}{\rm N}_x$ and $(\pi')^{-1}(V)=G\times_{G_y}{\rm N}_y$ are open in $M$
and $N$, respectively.
Let $f\colon V\to{\mathbb{R}}$ and assume $f\in {\rm C}^\infty(V)$. Then, by definition, the composed map
$f\circ\pi'\vert \colon G\times_{G_x}{\rm N}_y\to{\mathbb{R}}$ is smooth. But then
$$
f\circ \psi\vert_U\circ\pi_x =f\circ  \pi'_y\circ \tilde{\psi}    \colon {\rm N}_x\to{\mathbb{R}}
$$
is smooth. Consequently, $f\circ \psi\vert_U\circ \pi\vert\colon G\times_{G_x}{\rm N}_x
\to {\mathbb{R}}$ is smooth.
It follows that $\psi\vert_U$ is smooth in the sense on \cite{S}.
Since $x$ was chosen arbitrarily, it follows that also $\psi$ is smooth in the sense of \cite{S}.

We have proved the following: 

\begin{lemma}
\label{samatsileet}
Let $G$ be a  compact Lie group and let $M$ and $N$ be smooth $G$-manifolds. Assume the actions of $G$ on $M$ 
and $N$ are effective and almost free. 
Let $f\colon M/G\to N/G$ be a smooth orbifold  map.
Then $f$ is smooth in the sense of
\cite{S}.
\end{lemma}

Theorem \ref{isot.lift} and Lemma \ref{samatsileet} imply the following:

\begin{prop}
\label{ekv.nosto}
Let $X$ be a reduced smooth orbifold of dimension $n$, and let ${\rm Fr}(X)$ be the frame bundle of $X$.
Let $\bar{F}\colon X\times I\to X$ be a smooth orbifold isotopy starting at the identity. Then $\bar{F}$ is induced by a 
smooth ${\rm O}(n)$-equivariant isotopy $F\colon {\rm Fr}(X)\times I\to {\rm Fr}(X)$
starting at the identity. 
\end{prop}\qed

\section{Results}
\label{proofs}

\noindent  Let $G$ be a Lie group and let $H$ be a compact subgroup of $G$.  Let $M$ be a  proper
smooth $G$-manifold. Assume $M_H\not=\emptyset$. Since $G$ acts properly on $M$, it follows that
the normalizer $N(H)$  of $H$ in $G$
acts properly on $M_H$. Moreover, $N(H)_x=H$, for every $x\in M_H$. Let
$$
\phi\colon N(H)\times M_H\to M_H\times M_H,\,\,\, (k,x)\mapsto (kx,x).
$$
By Lemma \ref{G-paksu}, $N(H)/H\times M_H$ is diffeomorphic to 
$$
\phi\bigl(N(H)\times M_H\bigr)= \{ (kx,x)\mid k\in N(H), x\in M_H\}.
$$
Notice that
$$
\{(x,y)\in M_H\times M_H\mid Gx\not= Gy\}=(M_H\times M_H)\setminus 
\phi\bigl(N(H)\times M_H\bigr).
$$
Let $Y$ be a connected component of $M_H$.  Then $Y$ is a closed submanifold of
$M_H$ and the diffeomorphism $f$ of  Lemma \ref{G-paksu} takes
$N(H)/H\times Y$ onto the submanifold
$$
\{ (ky,y)\mid  k\in N(H), y\in Y\}=\phi\bigl(N(H)\times Y\bigr)
$$
of $M_H\times M_H$. Since $Y\times Y$ is a connected component of $M_H\times M_H$,
it follows that $\phi\bigl(N(H)\times Y\bigr)\cap(Y\times Y)$ is a submanifold of $Y\times Y$.

Let $N(H)_0$ denote the connected component of $N(H)$
containing the identity element. Then
$N(H)_0Y$ is connected,  which implies that $N(H)_0Y=Y$.  Thus
$\phi\bigl(N(H)_0\times Y\bigr)\subset Y\times Y$. Then
\begin{align*}
\dim\Bigl( \phi\bigl(N(H)\times Y\bigr)\cap(Y\times Y)\Bigr) &=
\dim \Bigl( \phi\bigl(N(H)_0\times Y\bigr)\cap(Y\times Y)\Bigr) \\
&= \dim \Bigl( \phi\bigl(N(H)_0\times Y\bigr)\Bigr)\\
&=\dim \Bigl( \bigl(N(H)_0H\bigr)/H\times Y\Bigr) \\
&=\dim \bigl( N(H)/H\times Y\bigr) \\
&= \dim\bigl( N(H)/H\bigr)+ \dim(Y).
\end{align*}
Notice that
$$
\{(x,y)\in Y\times Y\mid Gx\not= Gy\} =
(Y\times Y)\setminus \Bigl( \phi\bigl(N(H)\times Y\bigr)\cap(Y\times Y)\Bigr).
$$
Thus the space $\{(x,y)\in Y\times Y\mid Gx\not= Gy\}$
is connected if the codimension of $\phi\bigl(N(H)\times Y\bigr)\cap(Y\times Y)$ 
in $Y\times Y$ is at least $2$. This happens, if and only if,
$$
\dim\bigl( N(H)/H\bigr)\leq \dim(Y)-2.
$$

\medskip
\noindent{\it Proof of Theorem \ref{result1}.} We consider $n$-tuples $(x_1,\ldots, x_n)$
of points in $M$. There are 
finitely many compact subgroups $H$ of $G$ and finitely many connected components $Y_i$ of the sets $M_H$,
$1\leq i\leq k$, for some $k\leq n$, such that each of the points $x_1,\ldots, x_n$ lies in some connected
component $Y_i$, for some $i\in \{1,\ldots, k\}$.
We may assume that for all $i\not= j$ and for all $g\in G\setminus N(H)_0$, 
$Y_i\not= gY_j$. This is because for any $G$-equivariant map $f\colon M\to M$, the
condition $f(y_i)=z_i$ equals the condition $f(gy_i)=gz_i$. Thus the condition
about points in the connected components $gY_i$ can all be expressed by
considering only points in $Y_i$.
Just like in the proof of the corresponding result for orbifolds (\cite{PR}, Theorem 6) we begin by reordering
the points $x_j$ and $y_j$ according to which set $Y_i$ they lie in. Thus, up to  a permutation, we assume there exist numbers
$1=l_1< l_2\cdots < l_k\leq n$ such that for all $i\leq k-1$, $x_j\in Y_i$, if and only if $l_i\leq j\leq l_{i+1}-1$, and $x_j\in Y_k$, if and only if
$l_k\leq j\leq n$ (and analogously for the points $y_j$). 

Let  $M^n$ denote the $n$-fold cartesian product $M\times \cdots \times M$, on which $G$ 
acts diagonally. Let $M^{(n)}$ denote
the subspace of $M^n$ consisting of $n$-tuples of points whose orbits are pairwise distinct. Thus
$$
M^{(n)}=\{ (z_1,\ldots, z_n)\in M^n\mid Gz_i\not= Gz_j,\,\, {\rm if} \,\, i\not= j\}.
$$
Let $n_i=l_{i+1}-l_i$ for all $1\leq i\leq k-1$, and $n_k=n-l_k+1$. Then 
$$
(x_1,\ldots, x_n),(y_1,\ldots, y_n)\in Y_1^{n_1}\times\cdots\times Y_k^{n_k}.
$$
We show that $M^{(n)}\cap (Y_1^{n_1}\times\cdots\times Y_k^{n_k})$ is connected.  The {\it fat $G$-diagonal}  in
$Y_i^{n_i}$ is 
$$
\Delta^G_i=\{ (z_1,\ldots, z_{n_i})\in Y_i^{n_i}\mid  Gz_p= Gz_q\,\, {\rm for \,\, some}\,\, p\not= q\}.
$$
By assumption, if $Y_i$ is a connected component of $M_H$ and 
if $\dim\bigl(N(H)/H\bigr)=\dim(Y_i)-1$,
then $Y_i$  contains at most one of the points $x_1,\ldots, x_n$.
Thus, in this case,  $n_{i}=1$ and $\Delta^G_i=\emptyset$. 
(Notice that if $\dim(Y_i)=1$, then $\dim\bigl(N(H)/H\bigr)\in\{0,1\}$.
If $\dim(N(H)/H)=0$, then the equation 
$\dim\bigl(N(H)/H\bigr)=\dim(Y_i)-1$ is satisfied. 
If $\dim(N(H)/H)=1$, then $Y_i$ consists of just one orbit, and again, by assumption,
$Y_i$ can contain at most one of the points $x_1,\ldots, x_n$.)
If 
$\dim\bigl(N(H)/H\bigr)<\dim(Y_i)-1$, then $\Delta^G_i$ is a union of
submanifolds of $Y_i^{n_i}$ of codimension at least two. Since $Y_i^{n_i}$ is connected, this implies that
the complement $Y_i^{n_i}\setminus \Delta^G_i$ is connected. Similarly,
$$
M^{(n)}\cap (Y_1^{n_1}\times\cdots\times Y_k^{n_k})=
Y_1^{n_1}\times\cdots\times Y_k^{n_k}\setminus
\bigcup_{i=1}^k( Y_1^{n_1}\times\cdots\times \Delta^G_i\times\cdots
\times Y_k^{n_k})
$$
is connected. The above equality holds since we assumed $Y_i\not= gY_j$, for all $i\not= j$
and for all $g\in G\setminus N(H)_0$.

The group ${\rm Diff}^G_c(M)$ acts diagonally on $M^{(n)}$ and this action preserves each stratum
$M^{(n)}\cap (Y_1^{n_1}\times\cdots\times Y_k^{n_k})$.  We will show that each
${\rm Diff}^G_c(M)$-orbit is open in $M^{(n)}\cap (Y_1^{n_1}\times\cdots\times Y_k^{n_k})$.
Therefore, let $(z_1,\ldots, z_n)\in M^{(n)}\cap (Y_1^{n_1}\times\cdots\times Y_k^{n_k})$.
For each $z_j$, choose an open $G$-invariant neighborhood $U_j$ such that $U_i\cap U_j=\emptyset$, if $i\not= j$.
Then $U_1\times\cdots \times U_n\subset M^{(n)}$. 
By Proposition \ref{apulause}, each $z_i$ has a neighborhood
$V_i$ in $Y_i$ with the property 
that $GV_i\subset U_i$, and for each $w_i\in V_i$ there exists $f_i\in
{\rm Diff}^G_c(M)$ such that the $G$-compact support of the isotopy from the
identity to $f_i$ is contained in $U_i$
and $f_i(z_i)=w_i$. Let $g=f_1\circ\cdots\circ   f_n$. 
Since ${\rm supp}(f_i)\cap {\rm supp}(f_j)=\emptyset$, for all $i\not= j$, it follows that
$g(z_i)=w_i$, for all $1\leq i\leq n$. 
This means that the ${\rm Diff}^G_c(M)$-orbit  of
$(z_1,\ldots, z_n)$ is open in $M^{(n)}\cap (Y_1^{n_1}\times\cdots\times Y_k^{n_k})$.  Since
$M^{(n)}\cap (Y_1^{n_1}\times\cdots\times Y_k^{n_k})$ is connected, there is exactly one orbit. Thus there exists
$f\in {\rm Diff}^G_c(M)$ such that $f(x_1,\ldots, x_n)=(y_1,\ldots, y_n)$.
\qed

\begin{cor}
\label{cor1}
Let $G$ be a Lie group and let $M$ be a proper smooth $G$-manifold. 
Let
$n\in{\mathbb{N}}$, and let $(x_1,\ldots, x_n)$ and $(y_1,\ldots, y_n)$ be two $n$-tuples of points in
$M$. 
Assume the Conditions $(1)-(3)$ in Theorem \ref{result1} are satisfied.
Then there exists a diffeomorphism $f\colon M/G\to M/G$
isotopic to the identity through a compactly supported isotopy, such that $f(Gx_i)=Gy_i$ for all $1\leq i\leq n$.
\end{cor}

\begin{proof}
Since the $n$-tuples $(x_1,\ldots, x_n)$ and $(y_1,\ldots, y_n)$ satisfy the conditions of Theorem \ref{result1}, there is 
an equivariant $G$-compactly supported isotopy $F\colon M\times I\to M$ with $F_0={\rm id}\colon M\to M$ and
$F_1(x_i)=y_i$, for all $i\in\{1,\ldots, n\}$.
Then $F$ induces a compactly supported isotopy $\tilde{F}\colon (M/G)\times I\to M/G$, with $\tilde{F}_0=
{\rm id}\colon M/G\to M/G$ and $\tilde{F}_1(Gx_i)=Gy_i$ for all $i\in \{1,\ldots, n\}$. The isotopy $\tilde{F}$ is
smooth in the sense of \cite{S}, since it is induced by the equivariant isotopy $F$.
\end{proof}

\noindent{\it Proof of Theorem \ref{result2}.} Let $X$ be a reduced smooth  orbifold.
Using the frame bundle construction, we  identify $X$ with an orbit space $M/G$, where $G$ is a compact Lie
group acting smoothly on $M$ by an effective, almost free action.  Let $\pi\colon M\to M/G$ denote the natural
projection. For every $i\in \{1,\ldots, n\}$ there exists $\tilde{x}_i\in M$ with $\pi(\tilde{x}_i)=x_i$. Denote the
isotropy subgroup of $\tilde{x}_i$ by $H_i$, for every $i$. Let $Y_i$ be the connected component of $M_{H_i}$
containing $\tilde{x}_i$. Since $\Sigma(x_i)=\hat{\Sigma}(x_i)$ (see Section \ref{strata}), and by Lemma 
\ref{MHconncomp}, we may choose points $\tilde{y}_i\in M$ with $\pi(\tilde{y}_i)=y_i$ in such a way that 
$\tilde{y}_i\in Y_i$, for each $i$.

We now have two $n$-tuples $(\tilde{x}_1,\ldots, \tilde{x}_n)$ and $(\tilde{y}_1,\ldots, \tilde{y}_n)$ in $M$.
We will show that these $n$-tuples satisfy the Conditions $(1)-(3)$ in Theorem \ref{result1}. Condition $(1)$
holds by assumption, and Condition $(2)$ holds by the way the points $\tilde{y}_i$ were chosen. 

It remains to verify that Condition $(3)$ holds. Let $x\in M_H$, and let $Y$ be the connected component
of $M_H$ containing $x$. Then the normalizer $N(H)$ acts properly on $N(H)Y$. Let ${\rm N}$ be a linear slice
at $x$ in $N(H)Y$. Then
$$
{\rm T}_x\bigl(N(H)Y\bigr)\cong {\rm T}_{eH}\bigl( N(H)/H\bigr)\oplus {\rm T}_x({\rm N}).
$$
Since $N(H)_0Y=Y$, and since $N(H)Y$ is a union of some connected components of $M_H$, it follows that
\begin{align*}
\dim(Y)&=\dim\bigl(N(H)_0Y\bigr)\\
&=\dim\bigl(N(H)/H\bigr)+\dim({\rm N})\\
&>\dim\bigl(N(H)/H\bigr)+1,
\end{align*}
when $\dim({\rm N})>1$. Assume $\dim({\rm N})=1$. Then
$\dim\bigl(\pi(Y)\bigr)=1$ and, by assumption, $\pi(Y)$ can contain at most one of the
points $x_i$. Consequently, $Y$ can contain at most one of the points $\tilde{x}_i$. Thus Condition $(3)$ is
satisfied. The claim of the theorem can now be proved the same way Corollary \ref{cor1} was proved. 
\qed


\begin{thebibliography}{999}
\bibitem{ALR} A. Adem, J. Leida, Y. Ruan,
{\em Orbifolds and stringy topology,} 
Cambridge Tracts in Mathematics 171, Cambridge Univ. Press, 2007.
\bibitem{CF} P. E. Conner, E. E. Floyd,
{\em Differentiable periodic maps,}
Ergebnisse der Mathematik und ihrer Grenzgebiete, (N.F.), Band 33,
Academic Press, Inc., 1964.
\bibitem{D} G. Dragomir,
{\em  The stratification of singular locus and closed geodesics on orbifolds,}
arXiv: 1504.07157 (2015).
\bibitem{I} S. Illman,
{\em Every proper smooth action of a Lie group is equivalent to a real analytic action:
a contribution to Hilbert's fifth problem,} Prospects in topology (Princeton, NJ, 1994),
189--220, Ann. of Math. Stud., {\bf 138}, Princeotrn Univ. Press, Princeton, NJ, 1995.
\bibitem{IK} S. Illman, M. Kankaanrinta,
{\em A new topology for the set ${\rm C}^{\infty,G}(M,N)$ of $G$-equivariant smooth maps,}
Math. Ann. 316 (2000) 139--168.
\bibitem{K} M. Kankaanrinta,
{\em Equivariant collaring, tubular neighbourhood and gluing theorems for proper Lie group actions,}
Algebr. Geom. Topol.  7, (2007), 1--27.
\bibitem{MV} P. W. Michor, C. Vizman,
{\em n-transitivity of certain diffeomorphism groups,}
Acta Math. Univ. Comenian. (N. S.) 63 (1994), no. 2, 221--225.
 \bibitem{Pa} R. S. Palais,    
          {\em On the existence of slices for actions of non-compact Lie groups,}
          Ann. of Math. (2) {\bf 73} (1961), 295--323.    
 \bibitem{Pa2} R. S. Palais,
 {\em When proper maps are closed,} 
 Proc. Amer. Soc. {\bf 24} (1970), 835--836.                 
 \bibitem{PR} F. Pasquotto, T. O. Rot,
 {\em On the transitivity of the group of orbifold diffeomorphisms,}  
 Transform. Groups  28 (2023), no. 2,  973--985.
 \bibitem{P} M. J. Pflaum,
 {\em Analytic and geometric study of stratified spaces,} Lecture Notes in Mathematics, 1768,
 Springer--Verlag, Berlin, 2001.
 \bibitem{S} G. W. Schwarz,
 {\em Lifting smooth homotopies of orbit spaces,}
 Inst. Hautes \'Etudes Sci. Publ. Math. No. 51 (1980), 37--135.
\end{thebibliography}
\end{document}